\documentclass[12pt]{amsart}

\usepackage[usenames,dvipsnames]{xcolor}
\usepackage[most]{tcolorbox}
\usepackage{tabularx}
\usepackage{array}
\usepackage{colortbl}
\tcbuselibrary{skins}

\usepackage[latin1]{inputenc}

\usepackage[above]{placeins}

\usepackage{mathtools, amsthm, amssymb}

\usepackage{verbatim}
\usepackage{graphicx}
\usepackage{caption}

\usepackage{tikz}
\usetikzlibrary{calc}

\usepackage{cite}

\usepackage{hyperref} 

\makeatletter
\@ifclassloaded{beamer}{}{
\setlength{\textheight}{21.5cm} 
\setlength{\oddsidemargin}{.5cm}
\setlength{\evensidemargin}{.5cm} \setlength{\textwidth}{15.1cm}
\setlength{\topmargin}{-.6cm}
}
\makeatother

\usepackage[capitalize,nameinlink]{cleveref} 
\makeatletter
\@ifpackageloaded{cleveref}{
\theoremstyle{amsplain} 
}{
\theoremstyle{plain}
}
\@ifclassloaded{beamer}{}{
\newtheorem{theorem}{Theorem}[section]
\newtheorem{lemma}[theorem]{Lemma}
\newtheorem{corollary}[theorem]{Corollary}
\newtheorem{proposition}[theorem]{Proposition}

\theoremstyle{definition}

\newtheorem{remark}[theorem]{Remark}

\newtheorem{examples}[theorem]{Examples}

\newtheorem{question}[theorem]{Question}
}
\makeatother

\usepackage{pgffor}
\foreach \x in {C,D,E,G,I,N,Q,Z,D,R,S,T,U,W}{\expandafter\xdef\csname\x\x\endcsname{\noexpand\mathbb{\x}}}

\foreach \x in {A,B,C,D,E, F,H,I,J,K,P,S,T,U,V,X,Y,Z}{\expandafter\xdef\csname c\x\endcsname{\noexpand\mathcal{\x}}}

\foreach \x in {A,B,D,F,K,S,T,X}{\expandafter\xdef\csname m\x\endcsname{\noexpand\mathfrak{\x}}}

\DeclareMathOperator{\re}{Re} 
\DeclareMathOperator{\conv}{conv}
 \DeclareMathOperator{\tr}{Tr}

\renewcommand\emptyset{\varnothing}

\renewcommand\t{\theta} %

\numberwithin{equation}{section}

\newcommand\mytextwidth{\textwidth}  

\title{Higher Rank Numerical Ranges of Jordan-Like Matrices}


\author{Martín Argerami}
\address{Department of Mathematics and Statistics, University of Regina}
\email{argerami@uregina.ca}
\author{Saleh  Mustafa}

\date{\today}

\AtBeginDocument{
   \def\MR#1{}
}

\begin{document}

\maketitle

\begin{abstract}
We completely characterize the higher rank numerical range of the matrices of the form $J_n(\alpha)\oplus\beta I_m$, where $J_n(\alpha)$ is the $n\times n$ Jordan block with eigenvalue $\alpha$. Our characterization allows us to obtain concrete examples of several extreme properties of higher rank numerical ranges. 
\end{abstract}

\section{Introduction}

For a linear operator $T$ acting on a Hilbert space $\cH$, its \emph{numerical range} is the set 
\[
\Lambda_1(T)=\{\langle Tx,x\rangle:\ x\in\cH,\ \|x\|=1\}. 
\]
When $\cH$ is finite-dimensional, which will always be the case for us, it is easy to see that $\Lambda_1(T)$ is compact. A less obvious fact is that it is always convex: this is the famous Toeplitz--Hausdorff Theorem. The (closure of, in the infinite-dimensional case) the numerical range of $T$ always contains the spectrum $\sigma(T)$. The numerical range has applications in and is related to many areas, like matrix analysis, inequalities, operator theory, numerical analysis, perturbation theory, quantum computing, and others, see \cite{ando1973,2018BickelGorkin-Compressions,2007GauLi-Preserving,1995Kato-PerturbationTheory,2009KribsPasiekaLaforestRyanDaSilva-ResearchProblems,li2008canonical,2002LiTamWu-Nonnegative,1992Li-Inequalities,1993Spijker-Stability} for a few examples. We refer a reader who is not familiar with the numerical range to \cite[Chapter 1]{roger1994topics}.

Being such a well-known and important object, several generalizations of the numerical range have been considered, though we will only mention two of them. If we write 
\[
\Lambda_1(T)=\{\tr(TP):\ P\ \text{ is a projection of rank one}\}
\]
we get a generalization by taking different values for the rank of $P$; that way we get Halmos' $k$-numerical range \cite{Halmos-book}:
\[
W_k(T)=\{\tr(TP):\ P\ \text{ is a projection of rank }k\}.
\]
If we write 
\[
\Lambda_1(T)=\{\lambda\in\CC:\ \text{ there exists a rank-one projection $P$ with } PTP=\lambda P\}
\]
we obtain as a generalization the \emph{higher rank $k$-numerical range} \cite{choi2006higherB}:
\begin{equation}\label{equation: definition of Lambda k}
\Lambda_k(T)=\{\lambda\in\CC:\ \text{ there exists a rank-$k$ projection $P$ with } PTP=\lambda P\},
\end{equation}
that we consider in this paper. 
For a given $T$, we have $\Lambda_1(T)\supset\Lambda_2(T)\supset\cdots$ and each $\Lambda_k(T)$ is compact and convex. This last fact---convexity---is not obvious and was proven independently by Woerdeman \cite{Woerdeman2008-Convex} and Li-Sze \cite{li2008canonical} by very different means. 

Higher-rank numerical ranges have been calculated explicitly in some cases, but the list is fairly limited. The higher numerical range is invariant under unitary conjugation and respects translations---that is, $\Lambda_k(T+\beta I)=\beta+\Lambda_k(T)$---which expands a bit on whatever examples one has. For normal $T$ it was conjectured in \cite{choi2006higherA} and proven in \cite{li2008canonical} that
\[
\Lambda_k(T)=\bigcap_{\Gamma\subset\{\lambda_1,\ldots,\lambda_n\},\ |\Gamma|=n-k+1}\conv\Gamma,
\]
where $\lambda_1,\ldots,\lambda_n$ are the eigenvalues of $T$. 

The first case where higher rank numerical ranges of non-normal operators were calculated explicitly is \cite{GaayaHigherRank2012}, where the author shows that $\Lambda_k(T)$ is either a disk or empty whenever the $n\times n$ matrix $T$ is a  power of a shift. In \cite{2018AdamAretakiSpitkovsky-Elliptical} the authors determine the higher rank numerical ranges of direct sums of the form $\lambda I\oplus A_1\oplus \cdots\oplus A_n$, where the matrices $A_j$ are $2\times 2$, all with the same diagonal; this allows them---via unitary equivalence---to determine the higher numerical ranges of certain 2-Toeplitz tridiagonal matrices. In the cases where the structure of the chain $\Lambda_1(T),\ldots,\Lambda_n(T)$ is determined explicitly, its structure is fairly simple, going from a fixed type of area (a disk in \cite{GaayaHigherRank2012} and an ellipse in \cite{2018AdamAretakiSpitkovsky-Elliptical}) to the empty set. By contrast, the higher rank numerical ranges we find have more variety, see \cref{theorem: Omegak of T alpha beta m}. 

As in the aforementioned works, the convexity proof by Li--Sze  gives us the tool that we use  to calculate $\Lambda_k$ in our examples (a method derived from Li--Sze's formula \eqref{equation: numerical range as convex} is considered in \cite{ChienNakazato-Boundary}, but it does not look like it could be effectively used in our case). Recall the following well-known characterization of the numerical range: if $\lambda_1(T)$ denotes the largest eigenvalue of $T$, then by focusing on the convexity of the numerical range it is possible to prove that 
\begin{equation}\label{equation: numerical range as convex}
\Lambda_1(T)=\{\mu:\ \re e^{i\t}\mu\leq\lambda_1(\re e^{i\t} T),\ 0\leq\t\leq2\pi\}
\end{equation}
(see \cite[Theorem 1.5.12]{roger1994topics}). 
What Li and Sze showed is that that the equality \eqref{equation: numerical range as convex} extends naturally to the generalization \eqref{equation: definition of Lambda k}. Namely, 
\begin{theorem}[\cite{li2008canonical}]\label{theorem: li2008canonical}
Let $T\in M_n(\CC)$, $k\in\{1,\ldots,n\}$. Then 
\[
\Lambda_k(T)=\{\mu:\ \re e^{i\t}\mu\leq\lambda_k(\re e^{i\t} T),\ 0\leq\t\leq2\pi\}.
\]
\end{theorem}
This is very useful from a practical point of view, because the inequality $\re e^{i\t}\mu\leq\lambda_k(\re e^{i\t} T)$ describes a semi-plane in the complex plane, and one can sometimes plot or analyze the lines $\re e^{i\t}\mu=\lambda_k(\re e^{i\t} T)$ for each $\t$. 

The paper is organized as follows. In \cref{section: preliminaries} we develop some notation and discuss the sets that will arise in our description of higher rank numerical ranges. In \cref{section: calculations} we determine explicitly the higher rank numerical ranges of matrices of the form $J_n(\alpha)\oplus \beta I_m$. And in \cref{section: remarks} we consider some applications and relations with previous work. 

\section{Preliminaries}\label{section: preliminaries}

We begin by developing a bit of notation to express the sets that will arise as higher rank numerical ranges.

Our data consists of  $m,n\in\NN$ with $n\geq2$, $k\in\{1,\ldots,n+m\}$, and $\alpha,\beta\in\CC$. In terms of those numbers we will define angles $\phi_k,\psi_{k,m},\delta_k,\eta_{k,m}$, sets $D_k,C_{k,m}\subset\RR$ and $\widetilde D^{\vphantom B}_k, \widetilde C^{\vphantom B}_{k,m}\subset\CC$, and cones $R_{r,k}\subset\CC$ for some $r\geq0$.

Define
\[
\phi_k=\frac{k\pi}{n+1},\ \ \ \ \psi_{k,m}=\frac{(k-m)\pi}{n+1}. 
\]
The numbers $\cos\phi_k$ and $\cos\psi_{k,m}$ play an essential role in the statements and proofs to follow, so we encourage the reader to keep them in mind. In terms of these two numbers we define  two subsets of the real line, depending also on $\alpha,\beta$:  
\[
D_k=\{\t:\ |\beta-\alpha|\,\cos\t\leq\cos\phi_k\}
\]
and
\[
C_{k,m}=\begin{cases} \{\t:\ |\beta-\alpha|\,\cos\t>\cos\psi_{k,m}\},&\ k>m \\[0.3cm] \emptyset,&\ k\leq m
\end{cases}
\]
Note that we have  $-D_k=D_k$ and $-C_{k,m}=C_{k,m}$. These sets will only be relevant for $k\leq n/2$. When $n/2\geq k>m$ we have  $\psi_{k,m}<\phi_k<\pi$ and so $\cos\phi_k<\cos\psi_{k,m}$; from this it is clear that we always have $D_k\cap C_{k,m}=\emptyset$.

To characterize the sets $D_k$ and $C_{k,m}$ we will define two auxiliary angles, $\delta_k$ and $\eta_{k,m}$. First, let 
\[
\delta_k=\begin{cases} \arccos\left(\frac1{|\beta-\alpha|}\,\cos\phi_k\right),&\ |\beta-\alpha|\geq|\cos\phi_k| \ \text{ and } \beta\ne\alpha\\[0.3cm]  0,&\ \text{otherwise}
\end{cases}
\]

We remark that $0\leq\delta_k \leq\pi$, and that $\cos\phi_k\geq0$ if and only if $k\leq\tfrac{n+1}2$. 

\begin{lemma}\label{lemma: characterize Bk}
We have 
\[
D_k=
\begin{cases}
[\delta_k,2\pi-\delta_k]+2\pi\ZZ,&\ \delta_k> 0\\
[0,2\pi]+2\pi\ZZ,&\ \delta_k=0,\ k\leq \tfrac{n+1}2\\ 
\emptyset,&\ \delta_k=0,\ k>\tfrac{n+1}2
\end{cases}
\]
\end{lemma}
\begin{proof}
Assume first that $\delta_k>0$; in particular, $\beta\ne\alpha$. If $\t\in [\delta_k,2\pi-\delta_k]$, we have $\cos\t\leq\cos\delta_k$. That is, 
\[
\cos\t\leq\tfrac1{|\beta-\alpha|}\,\cos\phi_k,
\]
and so $\t\in D_k$. Conversely, if $\t\in D_k$ we have $\cos\t\leq\tfrac1{|\beta-\alpha|}\,\cos\phi_k=\cos\delta_k$, so $\t\in [\delta_k,2\pi-\delta_k]$. Thus $D_k=[\delta_k,2\pi-\delta_k]$. 

When $\delta_k=0$, we have $|\beta-\alpha|\leq|\cos\phi_k|$. If $k\leq\tfrac{n+1}2$, we have $\cos\phi_k\geq0$; then for any $\t$ we have $|\beta-\alpha|\cos\t\leq|\beta-\alpha|\leq|\cos\phi_k|=\cos\phi_k$, so $D_k=[0,2\pi]$. And if $k>\tfrac{n+1}2$, now $\cos\phi_k<0$; then $|\beta-\alpha|\cos\t\leq\cos\phi_k<0$ is impossible,  giving us $D_k=\emptyset$. 
\end{proof}

Our second auxiliary angle is 
\[
\eta_{k,m}=\begin{cases} \arccos\left(\frac1{|\beta-\alpha|}\cos\psi_{k,m}\right),&\ k>m,\ \beta\ne\alpha,\ \text{ and } \ |\beta-\alpha|\geq|\cos\psi_{k,m}| 
\\[0.3cm] 
0,&\ \text{otherwise}
\end{cases}
\]

\begin{lemma}\label{lemma: characterize Ck}
We have 
\[
C_{k,m}=\begin{cases}
[0,\eta_{k,m})\cup(2\pi-\eta_{k,m},2\pi],&\ \eta_{k,m}>0 \\[0.3cm]
\emptyset,&\ \eta_{k,m}=0,\ k\leq m\\[0.3cm]
\emptyset,&\ \eta_{k,m}=0,\ k> m,\ \cos\psi_{k,m}>0\\[0.3cm]
[0,2\pi],&\ \eta_{k,m}=0,\ k>m,\ \cos\psi_{k,m}<0
\end{cases}
\]
\end{lemma}
\begin{proof}
Consider first the case where $\eta_{k,m}>0$ (note that this includes the case $\cos\psi_{k,m}=0$). If $\t\in[0,\eta_{k,m})\cup(2\pi-\eta_{k,m},2\pi]$, we have $\cos\t>\cos\eta_{k,m}=\frac1{|\beta-\alpha|}\cos\psi_{k,m}$, so $\t\in C_{k,m}$. Conversely, if  $\t \in[\eta_{k,m},2\pi-\eta_{k,m}]$ we have  $\cos \t\leq \cos\eta_{k,m}=\frac1{|\beta-\alpha|}\cos\psi_{k,m}$, so $\t\not\in C_{k,m}$.

When $\eta_{k,m}=0$, we either have $k\leq m$, in which case $C_{k,m}=\emptyset$ by definition, or $k>m$. In this latter case we have $|\beta-\alpha|\leq|\cos\psi_{k,m}|$. If $\cos\psi_{k,m}>0$, then $|\beta-\alpha|\cos\t>\cos\psi_{k,m}$ is impossible, and so $C_{k,m}=\emptyset$; when $\cos\psi_{k,m}<0$, now $|\beta-\alpha|\cos\t\geq-|\beta-\alpha|\geq-|\cos\psi_{k,m}|=\cos\psi_{k,m}$. If the inequality is always strict, we have $C_{k,m}=[0,2\pi]$. Equality could only occur when $\cos\t=-1$ and $|\beta-\alpha|=-\cos\psi_{k,m}$; but this last equality, unless $\beta=\alpha$, implies $\eta_{k,m}=\pi$, contrary to our assumption of $\eta_{km}=0$. And if $\beta=\alpha$, $C_{k,m}=[0,2\pi]$ since $\cos\psi_{k,m}<0$. 
\end{proof}

Define, for each $k$, disjoint sets  $\widetilde D^{\vphantom B}_k, \widetilde E^{\vphantom B}_{k}\subset\CC$, with $\CC=\widetilde D^{\vphantom B}_k\cup \widetilde E^{\vphantom B}_{k}$, by 
\[
\widetilde D^{\vphantom B}_k=\{\mu\in\CC:\ \arg\mu\in D_k\},\ \ \ \widetilde E^{\vphantom B}_{k}=\{ \mu\in\CC:\ \arg\mu\not\in D_k\}.
\]
We will write $B_r(\lambda)$ for the \emph{closed} ball of radius $r$ centered at $\lambda$. We allow $r$ to be negative, in which case $B_r(\lambda)=\emptyset$. For $r\geq0$  denote  by $R_{r,k}$ the cone 
\[
R_{r,k}=\{\mu=x+iy\in\CC:\ x\leq r,\ \text{ and } \ (x-r)\cot\delta_k\leq y\leq (r-x)\cot\delta_k\}. 
\]

For a graphic description of these regions, we defer to \cref{remark: examples} and \cref{remark: new radius}.

\begin{lemma}\label{lemma: the cone}
Let $x,y\in\RR$, $r\geq0$. 
Assume that $0<\delta_k<\pi$.  
Then the following conditions are equivalent:
\begin{enumerate}
\item\label{lemma: the cone:1} $x\cos\t-y\sin\t\leq r\cos\t$ for all $\t\not\in D_k$;
\item\label{lemma: the cone:2} $x+iy\in R_{r,k}$; 
\item\label{lemma: the cone:3} $x\cos\delta_k\pm y\sin\delta_k\leq r\cos\delta_k$.
\end{enumerate}
\end{lemma}
\begin{proof}
\eqref{lemma: the cone:1}$\implies$\eqref{lemma: the cone:2} Since we only consider $\t\not\in D_k$ and $\delta_k>0$, by \cref{lemma: characterize Bk} we may assume that $-\delta_k<\t<\delta_k$. Assume first that $0\leq \t<\delta_k$, so that $\sin\t\geq0$. The case $\t=0$ (we   have $0\not\in D_k$ by the hypothesis $\delta_k>0$), gives us $x\leq r$. When  $\t\ne 0$, dividing the inequality $x\cos\t-y\sin\t\leq r\cos\t$ by $\sin\t$, we get $x\cot\t-y\leq r\cot\t$, which we rewrite as 
\begin{equation}\label{equation: inequality left}
(x-r)\cot\t\leq y,\ \ \ 0\leq\t<\delta_k.
\end{equation}
When $-\delta_k<\t<0$, we have $\sin\t<0$ and when we divide $x\cos\t-y\sin\t\leq r\cos\t$ by $\sin\t$, we get $x\cot\t-y\geq r\cot\t$, so 
\begin{equation}\label{equation: inequality right}
y\leq (x-r)\cot\t,\ \ \ -\delta_k<\t<0. 
\end{equation}
Replacing $\t$ with $-\t$ and using that $\cot(-\t)=-\cot\t$, the inequality \eqref{equation: inequality right} becomes
\begin{equation}\label{equation: inequality right2}
y\leq (r-x)\cot\t,\ \ \ 0<\t<\delta_k. 
\end{equation}
Combining \eqref{equation: inequality left} and \eqref{equation: inequality right2}, 
\begin{equation}\label{equation: inequality leftright}
(x-r)\cot\t\leq y\leq (r-x)\cot\t,\ \ \ 0<\t<\delta_k. 
\end{equation}
As the cotangent is decreasing on $(0,\pi)$, we have $\cot\delta_k\leq\cot\t$ if $0<\t<\delta_k$. From $(x-r)\leq0$, we obtain $(x-r)\cot\t\leq(x-r)\cot\delta_k$; since $y$ is at least as big as  $(x-r)\cot\t$ for all $\t\in(0,\delta_k)$, we get that $(x-r)\cot\delta_k\leq y$. Similarly, we have $(r-x)\cot\delta_k\leq(r-x)\cot\t$ for all $\t\in(0,\delta_k)$, so 
\[
(x-r)\cot\delta_k\leq y\leq(r-x)\cot\delta_k.
\]

\eqref{lemma: the cone:2}$\implies$\eqref{lemma: the cone:3} We may rewrite the inequality $(x-r)\cot\delta_k\leq y\leq(r-x)\cot\delta_k$ as the two inequalities
\[
x\cot\delta_k-y\leq r\cot\delta_k,\ \ \ x\cot\delta_k+y\leq r\cot\delta_k. 
\] 
Since $\sin\delta_k>0$   we can multiply by $\sin\delta_k $ to get 
\[
x\cos\delta_k-y\sin\delta_k\leq r\cos\delta_k,\ \ \ x\cos\delta_k+y\sin\delta_k\leq r\cos\delta_k. 
\] 

\eqref{lemma: the cone:3}$\implies$\eqref{lemma: the cone:1} We have $\sin\delta_k>0$ by hypothesis. Dividing by $\sin\delta_k$ we get 
\[
x\cot\delta_k\pm y\leq r\cot\delta_k,
\]
or
\begin{equation}\label{equation: lemma the cone 4}
(x-r)\cot\delta_k\leq\pm y. 
\end{equation}
It follows that $x-r$ is less than or equal both a non-negative and a non-positive number, so $x-r\leq0$. Now rewrite \eqref{equation: lemma the cone 4} as 
\[
(x-r)\cot\delta_k\leq y\leq(r-x)\cot\delta_k.
\]
If $0<\t<\delta_k$, using that the cotangent is decreasing and that $x-r\leq0$ we obtain 
\begin{equation}\label{equation: undoing cotangent 1}
(x-r)\cot\t\leq(x-r)\cot\delta_k\leq y,
\end{equation}
which we may write as $x\cos\t-y\sin\t\leq r\cos\t$ (since $\sin\t>0$). Similarly, when $-\delta_k<\t<0$, we have $(r-x)\cot\t<(r-x)\cot(-\delta_k)=-(r-x)\cot\delta_k$. Thus
\begin{equation}\label{equation: undoing cotangent 2}
y\leq(r-x)\cot\delta_k<-(r-x)\cot\t=(x-r)\cot\t,
\end{equation}
which is (after multiplying by $\sin\t$, which is negative) $x\cos\t-y\sin\t\leq r\cos\t$. Thus
\[
x\cos\t-y\sin\t\leq r\cos\t,\ \ \ \t\not\in D_k. \qedhere
\]
\end{proof}

\section{Matrices of the form $J_n(\alpha)\oplus \beta I_m$} \label{section: calculations}

As before, our data is  $m,n\in\NN$ with $n\geq2$, $k\in\{1,\ldots,m+n\}$, $\alpha,\beta\in\CC$. We denote by $J_n(\alpha)$ the $n\times n$ Jordan block with eigenvalue $\alpha$. Our goal is to calculate $\Lambda_k(J_n(\alpha)\oplus \beta I_m)$. For any $T\in M_n(\CC)$, we will denote by $\lambda_1(T),\ldots,\lambda_n(T)$ its eigenvalues in non-increasing order, counting multiplicities. 

Consider $T=J_n(\alpha)\oplus \beta I_m\in M_{n+m}(\CC)$. Let $\psi=\arg(\beta-\alpha)$. Then 
\[
T=\alpha I_{n+m}+e^{i\psi}T^0_{\alpha,\beta},\ \ \ \ \text{ where }\ \  T^0_{\alpha,\beta}=e^{-i\psi} J_n(0)\oplus |\beta-\alpha|\,I_m.
\]
By considering $T^0_{\alpha,\beta}$ we are translating and  rotating $T$ so that the eigenvalue of the Jordan block is zero, and the eigenvalue of the scalar part is real and non-negative. Because translations and rotations apply trivially to the higher-rank numerical range, we will analyze the operators $T^0_{\alpha,\beta}$. 

Our goal is to apply \cref{theorem: li2008canonical}, so we need to calculate $\lambda_k(\re e^{i\t}\,T^0_{\alpha,\beta})$.

\subsection{The case $k\leq n$}

\begin{lemma}\label{lemma: lambda_k. k less than n. m}
Let $T^0_{\alpha,\beta}=e^{-i\psi}J_n(0)\oplus |\beta-\alpha|I_m$, and $k\in\{1,\ldots,n\}$. Then 
\[
\lambda_k(\re e^{i\t}\,T^0_{\alpha,\beta})=\begin{cases}\cos\psi_{k,m},&\ \t\in C_{k,m}\\ |\beta-\alpha|\cos \t,&\ \t\in [0,2\pi]\setminus(D_k\cup C_{k,m})\\ \cos \phi_k,&\ \t\in D_k \end{cases}
\]
\end{lemma}
\begin{proof}
Since $T^0_{\alpha,\beta}$ is a block-diagonal sum of two matrices, its eigenvalues will be the union of the eigenvalues of each block. The only eigenvalue of 
\[
\re(e^{i\t}|\beta-\alpha| I_m)=|\beta-\alpha|\,\cos\t\,I_m
\]
is $|\beta-\alpha|\cos\t$, with multiplicity $m$. For $\re(e^{i\t} e^{-i\psi}J_n(0))=\re(e^{i(\t-\psi)}J_n(0))$, since unitary conjugation preserves the eigenvalues, we can apply the following well-known trick (it appears in \cite{haagerup-delaharpe1992}, though it was likely known before). Write $J_n(0)=\sum_{k=1}^{n-1}E_{k,k+1}$. Then 
\[
2\re(e^{i(\t-\psi)}J_n(0))=\sum_{k=1}^{n-1}e^{i(\t-\psi)}E_{k,k+1}+e^{-i(\t-\psi)}E_{k+1,k}. 
\]
Now we conjugate with the diagonal unitary $\sum_{j=1}^n e^{ij(\t-\psi)}E_{jj}$: 
\begin{align*}
\sum_{j=1}^n &e^{ij(\t-\psi)}E_{jj}\left(\sum_{k=1}^{n-1}e^{i(\t-\psi)}E_{k,k+1}+e^{-i(\t-\psi)}E_{k+1,k}\right)\sum_{j=1}^n e^{-ij(\t-\psi)}E_{jj}\\ 
&
=\sum_{k=1}^{n-1}e^{ik(\t-\psi)} e^{i(\t-\psi)}e^{-i(k+1)(\t-\psi)} E_{k,k+1}+e^{i(k+1)(\t-\psi)} e^{-i(\t-\psi)}e^{-ik(\t-\psi)} E_{k+1,k}\\
&=\sum_{k=1}^{n-1}  E_{k,k+1}+  E_{k+1,k}=J_n(0)+J_n(0)^*=2\re(J_n(0)). 
\end{align*}
Thus the eigenvalues of $\re(e^{i(\t-\psi)}J_n(0))$ are the same as those of $\re(J_n(0))$, and these are well-known to be $\{\cos \frac{j\pi}{n+1}:\ j=1,\ldots,n\}$; this can be seen by working explicitly with the eigenvectors 
\[
\xi_k=\left(\sin \frac{k\pi}{n+1},\sin \frac{2k\pi}{n+1},\ldots,\sin \frac{nk\pi}{n+1}\right).
\]
The above calculation is mentioned explicitly in \cite{haagerup-delaharpe1992}, where they mention that it was known to Lagrange. The eigenvalues indeed appear in \cite[Page 76]{de1759recherches}, although his argument does not seem to be as clear as Haagerup--de La Harpe's. 

Now we know that the eigenvalues of $\re (e^{i\t}\,T^0_{\alpha,\beta})$ are $|\beta-\alpha|\cos\t$ ($m$ times) and $\{\cos \frac{j\pi}{n+1}:\ j=1,\ldots,n\}$. These last $n$ are already in non-increasing order. Remember that our goal is to find the $k^{\rm th}$ entry in the list. 

Consider first the case $k\leq m$, where $C_{k,m}=\emptyset$. If $\t\not\in D_k$ then $|\beta-\alpha|\cos\t>\cos\phi_k$; this implies that the $m$ instances of  $|\beta-\alpha|\cos\t$ appear in the (ordered) list of eigenvalues of $\re (e^{i\t}\,T^0_{\alpha,\beta})$ at most after $\cos \frac{(k-1)\pi}{n+1}$. As $k\leq m$, the $k^{\rm th}$ largest eigenvalue is then $|\beta-\alpha|\cos\t$. When $\t\in D_k$, we now have $|\beta-\alpha|\cos\t\leq\cos\phi_k$, so the first $k$ elements in the ordered list of eigenvalues are $\{\cos \frac{j\pi}{n+1}:\ j=1,\ldots,k\}$. Thus the $k^{\rm th}$ eigenvalue is $\cos\phi_k$. 

When $m<k\leq n$, the situation is a bit different, since now $C_{k,m}\ne\emptyset$. When $\t\in C_{k,m}$, we have $|\beta-\alpha|\cos\t>\cos\psi_{k,m}$. 
So the $m$ elements $|\beta-\alpha|\cos\t$ appear, in the list of eigenvalues,  before $\cos\psi_{k,m}$;  the list of eigenvalues looks like
\[
\cos\tfrac{\pi}{n+1},\ldots,\cos\tfrac{j\pi}{n+1},
\overbrace{|\beta-\alpha|\cos\t,\ldots,|\beta-\alpha|\cos\t}^{m\ \text{ times}},
\cos\tfrac{(j+1)\pi}{n+1},\ldots,\cos\tfrac{(k-m)\pi}{n+1},\ldots
\]
As the $m$ equal entries will always appear before $\cos\frac{(k-m)\pi}{n+1}$,  the $k^{\rm th}$ eigenvalue is $\cos\frac{(k-m)\pi}{n+1}=\cos\psi_{k,m}$. When $\t\in [0,2\pi]\setminus(D_k\cup C_{k,m})$, the $m$ eigenvalues $|\beta-\alpha|\cos\t$ sit somewhere between $\cos\frac{(k-m)\pi}{n+1}$ and $\cos \frac{k\pi}{n+1}$. Since there are at most $k-1$ elements of the form $\cos \frac{j\pi}{n+1}$ above the $m$ elements $|\beta-\alpha|\cos\t$ in the list, now the $k^{\rm th}$ eigenvalue is $|\beta-\alpha|\cos\t$. Finally, when $\t\in D_k$, the first $k$ eigenvalues in the list are $\cos \frac{j\pi}{n+1}$, $j=1,\ldots,k$, so the $k^{\rm th}$ element in the list is $\cos \frac{k\pi}{n+1}=\cos\phi_k$. 
\end{proof}

\begin{proposition}\label{theorem: Lambda_k of Jordan plus I m}
Let $T^0_{\alpha,\beta}=e^{-i\psi}J_n(0)\oplus |\beta-\alpha|I_m$, and $k\leq n$. Then   
\[
\Lambda_k(T^0_{\alpha,\beta})=\left(\vphantom{R_\beta^\beta}\widetilde D^{\vphantom B}_k\cap B_{\cos\phi_k}(0)\right)
\cup
\left(\vphantom{R_\beta^\beta}\widetilde E^{\vphantom B}_{k}\cap R_{|\beta-\alpha|,k}\cap X\right),
\]
where 
\[
X=
\begin{cases}
\CC,&\ \text{ if }\ k\leq m\ \text{ or }\ C_{k,m}=\emptyset\\[0.3cm] 
B_{\cos\psi_{k,m}}(0),&\ \text{ if }\ k>m,\ C_{k,m}\ne\emptyset
\end{cases} 
\]
\end{proposition}
\begin{proof}

We consider first the case $k\leq m$ or $C_{k,m}=\emptyset$; in both cases we have $C_{k,m}=\emptyset$. Throughout the proof, we will use \cref{theorem: li2008canonical} and \cref{lemma: lambda_k. k less than n. m} repeatedly. 

Suppose first that $\mu\in\Lambda_k(T^0_{\alpha,\beta})$. We write $\mu=|\mu|e^{i\xi}=x+iy$. We split in two complementary cases:
\begin{itemize}
\item $\mu\in \widetilde D^{\vphantom B}_k$. So $\xi=\arg\mu\in D_k$. We have, for all $\t\in D_k$, 
\[
|\mu|\,\cos(\xi+\t)=\re e^{i\t}\mu\leq\lambda_k(\re e^{i\t} T^0_{\alpha,\beta})=\cos\phi_k,\ \ \ \t\in D_k.
\]
As $D_k=-D_k$, we have that $-\xi\in D_k$, so 
\[
|\mu|=|\mu|\,\cos(\xi-\xi)\leq\cos\phi_k.
\]
That is,  $\mu\in B_{\cos\phi_k}(0)$. 

\item $\mu\in \widetilde E^{\vphantom B}_{k}$. 
For all $\t\not\in D_k$,
\[
x\cos\t-y\sin\t=\re e^{i\t}\mu\leq\lambda_k(\re e^{i\t} T^0_{\alpha,\beta})=|\beta-\alpha|\,\cos\t.
\]
By \cref{lemma: the cone}, $\mu=x+iy\in R_{|\beta-\alpha|,k}$ and thus $\mu\in \widetilde E^{\vphantom B}_{k}\cap R_{|\beta-\alpha|,k}$. 

\end{itemize}
\medskip

Now, for the converse, we also consider two complementary cases:
\begin{itemize}
\item $\mu\in \widetilde D^{\vphantom B}_k\cap B_{\cos\phi_k}(0)$. We have $\xi=\arg\mu\in D_k$ and $|\mu|\leq\cos\phi_k$. Then, for every $\t\in D_k$, 
\begin{equation}\label{equation: theta in B}
\re e^{i\t}\mu=|\mu|\cos(\xi+\t)\leq|\mu|\leq\cos\phi_k=\lambda_k(\re e^{i\t} T^0_{\alpha,\beta});
\end{equation}
and, for $\t\not\in D_k$,  
\begin{equation}\label{equation: theta not in B}
\re e^{i\t}\mu\leq|\mu|\leq\cos\phi_k<|\beta-\alpha|\,\cos\t=\lambda_k(\re e^{i\t} T^0_{\alpha,\beta}).
\end{equation}
Now \eqref{equation: theta in B} and \eqref{equation: theta not in B} together imply that $\mu\in\Lambda_k(T^0_{\alpha,\beta})$. 

\medskip

\item $\mu\in \widetilde E^{\vphantom B}_{k}\cap R_{|\beta-\alpha|,k}$. So $\xi=\arg\mu\not\in D_k$. For any $\t\not\in D_k$, and using \cref{lemma: the cone}, 
\[
\re(e^{i\t}\mu)=x\cos\t-y\sin\t\leq|\beta-\alpha|\,\cos\t=\lambda_k(\re e^{i\t} T^0_{\alpha,\beta}). 
\]
When $\t\in D_k$, by \cref{lemma: characterize Bk} the distance between $\t$ and $\xi$ is minimized at $\delta_k$ (if $0\leq\xi\leq \pi$), or at $-\delta_k$ (if $\pi<\xi<2\pi$). Thus, using again \cref{lemma: the cone},
\begin{align*}
\re(e^{i\t}\mu)&=|\mu|\cos(\xi+\t)\leq|\mu|\cos(\xi\pm\delta_k)
=x\cos\delta_k\mp y\cos\delta_k\\
&\leq |\beta-\alpha|\,\cos\delta_k
=\cos\phi_k=\lambda_k(\re e^{i\t} T^0_{\alpha,\beta}).
\end{align*}
So $\mu\in\Lambda_k(T^0_{\alpha,\beta})$. 
\end{itemize}

When $k>m$, the above proof still applies, with the only exception of the case where  $\mu\in\Lambda_k(T^0_{\alpha,\beta})$ and $\mu\in \widetilde E^{\vphantom B}_{k}$---that is, the second bullet above. We still get that $\mu\in R_{|\beta-\alpha|,k}$, but now we can consider whether $\xi\in C_{k,m}$ or not. Recall that $\xi\not\in D_k$ since $\mu\in\widetilde E^{\vphantom B}_{k}$. If $\xi\in C_{k,m}$, then we also have $-\xi\in C_{k,m}$. Then
\[
|\mu|=|\mu|\cos(\xi-\xi)=\re(e^{-i\xi}\mu)\leq\lambda_k(\re e^{-i\xi}T^0_{\alpha,\beta})=\cos\psi_{k,m}. 
\]
When $\xi\not\in (D_k\cup C_{k,m})$, we have $\eta_{k,m}\leq\xi\leq\delta_k$ or $2\pi-\delta_k\leq\xi\leq2\pi-\eta_{k,m}$ (\cref{lemma: characterize Bk,lemma: characterize Ck}). Then 
\begin{align*}
|\mu|&=|\mu|\cos(\xi-\xi)=\re(e^{-i\xi}\mu)\leq\lambda_k(\re e^{-i\xi}T^0_{\alpha,\beta})\\ 
&=|\beta-\alpha|\cos\xi\leq|\beta-\alpha|\cos\eta_{k,m}=\cos\psi_{k,m}. 
\end{align*}
So in both cases $\mu\in B_{\cos\psi_{k,m}}(0)$ and we are done. 
\end{proof}

Now we can gather some insight on the shape of $\Lambda_k(T^0_{\alpha,\beta})$ when $k\leq n/2$ (the case $k>n/2$ is always somewhat degenerate, as we will see). When $k\leq m$, the convex set $\Lambda_k(T^0_{\alpha,\beta})$ is the union of two sets:  $\widetilde D^{\vphantom B}_k\cap B_{\cos\phi_k}(0)$ and $\widetilde E^{\vphantom B}_{k}\cap R_{|\beta-\alpha|,k}$. The former is a circular sector, while the latter is the intersection of two cones. We refer the reader to Figures \ref{n4i4k1} and \ref{n4i4k2} to visualize the shape. What is not obvious from the description in \cref{theorem: Lambda_k of Jordan plus I m} is how the two regions are joined. It turns out that the edges coming from the corner point $|\beta-\alpha|$ (or $\beta$ in the general case) are tangent to the disk $B_{\cos\phi_k}$ precisely at the point where they intersect the edges of $\widetilde E^{\vphantom B}_{k}$. That is what we prove in the next two propositions. 

\begin{proposition}\label{proposition: understanding the shape 2}
If $k\leq n/2$ and $|\beta-\alpha|\leq\cos\phi_k$, then $\Lambda_k(T^0_{\alpha,\beta})=B_{\cos\phi_k}(0)$. That is, if the distance between the eigenvalue of the scalar block and eigenvalue of the Jordan block is less than $\cos\phi_k$, the $k^{\rm th}$ higher rank numerical range is a disk.  
\end{proposition}
\begin{proof}
The hypothesis $|\beta-\alpha|\leq\cos\phi_k$ guarantees that $D_k=[0,2\pi]$ and so $\widetilde D^{\vphantom B}_k=\CC$; thus $\widetilde E^{\vphantom B}_{k}=\emptyset$ and the result follows from \cref{theorem: Lambda_k of Jordan plus I m}.
\end{proof}

\begin{proposition}\label{proposition: understanding the shape}
If $k\leq n/2$ and $|\beta-\alpha|>\cos\phi_k$, then  
\[
\left(\widetilde D^{\vphantom B}_k\cap B_{\cos\phi_k}(0)\right)\cup \left(\widetilde E^{\vphantom B}_{k}\cap R_{|\beta-\alpha|,k}\cap X\right)=B_{\cos\phi_k}(0)\cup \left(\widetilde E^{\vphantom B}_{k}\cap R_{|\beta-\alpha|,k}\cap X\right),
\]
where $X=\CC$ if $k\leq m$, and $X=B_{\cos\psi_{k,m}}(0)$ if $k>m$. 
Moreover, the lines $x\cos\delta_k\pm y\sin\delta_k=|\beta-\alpha|\cos\delta_k$ that form the boundary of $R_{|\beta-\alpha|,k}$ are tangent to the circle $x^2+y^2=\cos^2\phi_k$ (that is, to the boundary of $B_{\cos\phi_k}(0)$). 
\end{proposition}
\begin{proof}
The condition $k\leq n/2$ guarantees that $\cos\phi_k>0$. When $k>m$ (the only case where $\psi_{k,m}$ matters) we always have $\cos\psi_{k,m}>\cos\phi_k$ (since $0<k-m<k\leq n/2$). So whenever $z\in B_{\cos\phi_k}(0)$, we have $z\in B_{\cos\psi_{k,m}}(0)$. 

If $z\in \widetilde E^{\vphantom B}_{k}\cap B_{\cos\phi_k}(0)$, we have $z=re^{i\xi}$ with $0\leq r\leq \cos\phi_k$ and $-\delta_k<\xi<\delta_k$. Then (recall that $\delta_k<\pi/2$ from $k\leq n/2$)
\[
r\cos\xi\cos\delta_k\mp r\sin\xi\sin\delta_k=r\cos(\xi\pm\delta_k)\leq\cos\phi_k=|\beta-\alpha|\cos\delta_k,
\]
and so by \cref{lemma: the cone} we have $z=r\cos\xi+ ir\sin\xi\in R_{|\beta-\alpha|,k}$. Thus 
\[
\left(\widetilde D^{\vphantom B}_k\cap B_{\cos\phi_k}(0)\right)\cup \left(\widetilde E^{\vphantom B}_{k}\cap R_{|\beta-\alpha|,k}\cap X\right)\supset B_{\cos\phi_k}(0)\cup \left(\widetilde E^{\vphantom B}_{k}\cap R_{|\beta-\alpha|,k}\cap X\right),
\]
which is the nontrivial inclusion. 

Now for the lines, let us look the intersection of each of the two lines $x\cos\delta_k\pm y\sin\delta_k=|\beta-\alpha|\cos\delta_k$ and the circle $x^2+y^2=\cos^2\phi_k$. Recall that $|\beta-\alpha|\cos\delta_k=\cos\phi_k$. A point in the circle has coordinates $(\cos\phi_k\cos\t,\cos\phi_k\sin\t)$ for some $\t$. If this point belongs to the line $x\cos\delta_k-y\sin\delta_k=\cos\phi_k$, we get 
\[
\cos\phi_k=\cos\phi_k\cos\t\cos\delta_k-\cos\phi_k\sin\t\sin\delta_k=\cos\phi_k\cos(\t+\delta_k). 
\]
The hypothesis $k\leq n/2$ guarantees that $\cos\phi_k\ne0$, so we get
\[
1=\cos(\t+\delta_k)
\]
and thus $\t=-\delta_k$. The slope of the line is $\cot\delta_k$; the slope of the circle at the point $(\cos\phi_k\cos(-\delta_k),\cos\phi_k\sin(-\delta_k))$ is $-1/\tan(-\delta_k)=\cot\delta_k$, so the line is tangent to the circle. 

The other line gives $\t=\delta_k$, and a similar computation shows that it is also tangent to the circle. 
\end{proof}

\subsection{The case $k>n$} In this case we have $\phi_k\geq\pi/2$. 
Recall that $C_{k,m}=\emptyset$ if $k\leq m$.

\begin{lemma}\label{lemma: lambda_k. k greater than m}
If $T^0_{\alpha,\beta}=e^{-i\psi}J_n(0)\oplus|\beta-\alpha|I_m$ and $k>n$, then 
\[
\lambda_k(\re e^{i\t}T^0_{\alpha,\beta})=\begin{cases} 
\cos\psi_{k,m},&\ \t\in C_{k,m}\\[0.3cm]
|\beta-\alpha|\cos\t,&\ \t\not\in C_{k,m}
\end{cases}
\]
As a consequence, 
\[
\Lambda_k(T^0_{\alpha,\beta})=\begin{cases} \{|\beta-\alpha|\},&\ k\leq m\\[0.3cm]
\{|\beta-\alpha|\},&\ k>m\ \text{ and }\ |\beta-\alpha|\leq\cos\psi_{k,m}\\[0.3cm]
\emptyset,&\ k>m\ \text{ and }\ |\beta-\alpha|>\cos\psi_{k,m}
\end{cases}
\]
\end{lemma}
\begin{proof}
If $\t\in C_{k,m}$, this means by definition that that $k>m$ and $|\beta-\alpha|\cos\t>\cos\psi_{k,m}$. So the first $k$ eigenvalues of $\re(e^{i\t}T^0_{\alpha,\beta})$ will be
\[
\cos\tfrac\pi{n+1},\ldots,\cos\tfrac{{(j-1)}\pi}{n+1},\overbrace{|\beta-\alpha|\cos\t,\ldots,|\beta-\alpha|\cos\t}^{m\ \text{times}},\cos\tfrac{j\pi}{n+1},\ldots,\cos\tfrac{(k-m)\pi}{n+1},
\]
where $j\in\{1,\ldots,n-m-1\}$. 
Thus the $k^{\rm th}$ eigenvalue is $\cos\tfrac{(k-m)\pi}{n+1}=\cos\psi_{k,m}$. When $\t\not\in C_{k,m}$, the $m$ numbers $|\beta-\alpha|\cos\t$ will sit after $\cos\psi_{k,m}$; that is the list looks like 
\[
\cos\tfrac\pi{n+1},\ldots,\cos\tfrac{{(j-1)}\pi}{n+1},\overbrace{|\beta-\alpha|\cos\t,\ldots,|\beta-\alpha|\cos\t}^{m\ \text{times}},\cos\tfrac{j\pi}{n+1},\ldots,\cos\tfrac{k\pi}{n+1},
\]
where now $j\in\{k-m+1,\ldots,k\}$. Thus the $k^{\rm th}$ eigenvalue will always be $|\beta-\alpha|\cos\t$. That is,
\[
\lambda_k(\re e^{i\t}T^0_{\alpha,\beta})=\begin{cases} 
\cos\psi_{k,m},&\ \t\in C_{k,m}\\[0.3cm]
|\beta-\alpha|\cos\t,&\ \t\not\in C_{k,m}
\end{cases}
\]
Now if $\mu=x+iy\in\Lambda_k(T^0_{\alpha,\beta})$, we have by the above 
\begin{equation}\label{lemma: lambda_k. k greater than m:1}
x\cos\t-y\sin\t=\re(e^{i\t}\mu)\leq\lambda_k(\re(e^{i\t}T^0_{\alpha,\beta}))=\cos\psi_{k,m},\ \ \t\in C_{k,m},
\end{equation}
and
\begin{equation}\label{lemma: lambda_k. k greater than m:2}
x\cos\t-y\sin\t=\re(e^{i\t}\mu)\leq\lambda_k(\re(e^{i\t}T^0_{\alpha,\beta}))=|\beta-\alpha|\cos\t,\ \ \t\not\in C_{k,m}. 
\end{equation}
Suppose that $|\beta-\alpha|>\cos\psi_{k,m}$. Then $0\in C_{k,m}$; we get from \eqref{lemma: lambda_k. k greater than m:1}, with $\t=0$,  that $x\leq\cos\psi_{k,m}$. If $\pi\not\in C_{k,m}$, we get from \eqref{lemma: lambda_k. k greater than m:2} that $-x\leq-|\beta-\alpha|$; so $x\geq|\beta-\alpha|>\cos\psi_{k,m}$ and we get a contradiction. And if $\pi\in C_{k,m}$, now $C_{k,m}=[0,2\pi]$ and so \eqref{lemma: lambda_k. k greater than m:1} gives us $0\leq|\mu|\leq\cos \psi_{k,m}$; but then, using that $\pi\in C_{k,m}$, 
$-|\beta-\alpha|=|\beta-\alpha|\cos\pi>\cos\psi_{k,m}$ giving us $|\beta-\alpha|<-\cos\psi_{k,m}\leq0$, a contradiction. Thus $\Lambda_k(T^0_{\alpha,\beta})=\emptyset$ when $|\beta-\alpha|>\cos\psi_{k,m}$.

If $|\beta-\alpha|\leq\cos\psi_{k,m}$, then $C_{k,m}=\emptyset$, so \eqref{lemma: lambda_k. k greater than m:2} applies for all $\t$. Taking $\t=\pm\pi/2$, we get $\pm y\leq 0$, so $y=0$. Then with $\t=0$ and $\t=\pi$ we get 
$x\leq|\beta-\alpha|$ and $x\geq|\beta-\alpha|$, so $x=|\beta-\alpha|$. Now \eqref{lemma: lambda_k. k greater than m:2} reads $|\beta-\alpha|\cos\t\leq|\beta-\alpha|\cos\t$, which obviously holds for all $\t$ and so $\Lambda_k(T^0_{\alpha,\beta})=\{|\beta-\alpha|\}$. 

When $k\leq m$, we have $C_{k,m}=\emptyset$ and the previous paragraph applies. 
\end{proof}

We can now prove our main result. We remark that the area $B_{\cos\phi_k}(0)\cup (\widetilde E^{\vphantom B}_{k}\cap R_{|\beta-\alpha|,k})$ below is precisely the sector $\{\mu\in\CC:\ \re e^{i\t}\mu\leq r,\ \delta_k\leq\t\leq2\pi-\delta_k\}$; we use the former notation to avoid using $\delta_k$ in the statements. 

\begin{proposition}\label{proposition: Omegak of T alpha beta m}
Let $T^0_{\alpha,\beta}=e^{i\psi}J_n(0)\oplus |\beta-\alpha|I_m$. Let $k\in\{1,\ldots,n+m\}$. Then $\Lambda_k(T^0_{\alpha,\beta})$ is as in the following table:  \\[0.1cm]

\tcbset{tab2/.style={enhanced,fonttitle=\bfseries,fontupper=\normalsize\sffamily,
colback=yellow!10!white,colframe=red!50!black,colbacktitle=Salmon!30!white,
coltitle=black,center title}}

\renewcommand{\arraystretch}{1.5}
\newcounter{lista}
\setcounter{lista}{1}

\newsavebox{\tablaprop}
\savebox{\tablaprop}{
\begin{tabular}{rll}
\multicolumn{1}{c}{ }&\multicolumn{1}{c}{\text{$\Lambda_k(T^0_{\alpha,\beta})$}} &\multicolumn{1}{c}{ Condition  }       \\\hline  \\[-0.1cm]
\arabic{lista}\stepcounter{lista}& $B_{\cos\phi_k}(0)$ & $\ 1\leq k\leq\frac n2,\   |\beta-\alpha|\leq\cos\phi_k$ \ \\[0.3cm] \hline
\arabic{lista}\stepcounter{lista}& $B_{\cos\phi_k}(0)\cup (\widetilde E^{\vphantom B}_{k}\cap R_{|\beta-\alpha|,k})$&$\ 1\leq k\leq\frac n2, \ k\leq m,\ |\beta-\alpha|>\cos\phi_k$ \ \\[0.3cm]  \hline
\arabic{lista}\stepcounter{lista}& $B_{\cos\phi_k}(0)\ \cup \left(\vphantom{R_\beta^\beta}\widetilde E^{\vphantom B}_{k}\cap R_{|\beta-\alpha|,k}\cap B_{\cos\psi_{k,m}}(0)\right)$ & $\ 1\leq k\leq\frac n2,\  k> m,\ |\beta-\alpha|>\cos\phi_k$ \ \\[0.3cm]  \hline

\arabic{lista}\stepcounter{lista}& $[0,|\beta-\alpha|]$ & $\ k=\frac{n+1}2\leq m$,\ \text{ or }  \ \\[0.3cm] 
&& $\ k=\frac{n+1}2> m,\ |\beta-\alpha|\leq\cos\psi_{k,m}$ \ \\[0.3cm]  \hline
\arabic{lista}\stepcounter{lista}& $[0,\cos\psi_{k,m}]$ & $\ k=\frac{n+1}2> m, \ |\beta-\alpha|>\cos\psi_{k,m}$ \ \\[0.3cm]  \hline
\arabic{lista}\stepcounter{lista}& $\{|\beta-\alpha|\}$ & $\ \frac{n+1}2<k\leq m$,\ \text{or} \ \\[0.3cm]
&&$\ \frac{n+1}2<k,\ k> m,\ |\beta-\alpha|\leq\cos\psi_{k,m}$ \ \\[0.3cm] \hline
\arabic{lista}\stepcounter{lista}& $\emptyset$ & $\ \frac{n+1}2<k,\ k> m,\ |\beta-\alpha|>\cos\psi_{k,m}$ \ \\[0.3cm]
\end{tabular}
} 
\begin{center}
\resizebox{14cm}{!}{
\begin{tcolorbox}[tab2,tabularx={rll},title={$T^0_{\alpha,\beta}=e^{-i\psi}J_n(0)\oplus |\beta-\alpha|I_m$},boxrule=0.5pt,width=\wd\tablaprop]
\usebox{\tablaprop}
\end{tcolorbox}
} 
\end{center}
\end{proposition}
\begin{proof}
We go through the conditions in the table.
\setcounter{lista}{1}
\begin{enumerate}
\item[(\arabic{lista})]\stepcounter{lista} $k\leq\frac n2,\   |\beta-\alpha|\leq\cos\phi_k$: By \cref{proposition: understanding the shape 2},
\[
\Lambda_k(T^0_{\alpha,\beta})=B_{\cos\phi_k}(0).
\]
\item[(\arabic{lista})]\stepcounter{lista} $k\leq\frac n2, \ k\leq m,\ |\beta-\alpha|>\cos\phi_k$: Here $\phi_k<\pi/2$, so $\cos\phi_k>0$. By \cref{proposition: understanding the shape},
\[
\Lambda_k(T^0_{\alpha,\beta})=B_{\cos\phi_k}(0)\cup (\widetilde E^{\vphantom B}_{k}\cap R_{|\beta-\alpha|,k}).
\]

\item[(\arabic{lista})]\stepcounter{lista} $k\leq\frac n2, \ k> m,\ |\beta-\alpha|>\cos\phi_k$: Again $\phi_k<\pi/2$, so $\cos\phi_k>0$. By \cref{proposition: understanding the shape},
\[
\Lambda_k(T^0_{\alpha,\beta})=B_{\cos\phi_k}(0)\cup (\widetilde E^{\vphantom B}_{k}\cap R_{|\beta-\alpha|,k}\cap B_{\cos\psi_{k,m}}(0)).
\]

\item[(\arabic{lista})]\stepcounter{lista}\label{item: step 4} $k=\frac{n+1}2\leq m$: now $\cos\phi_k=\cos\pi/2=0$, so $\delta_k=\pi/2$ and  $D_k=[\pi/2,3\pi/2]$. From \cref{theorem: Lambda_k of Jordan plus I m} we have 
\[
\Lambda_k(T^0_{\alpha,\beta})=\left(\vphantom{R_\beta^\beta}\widetilde D^{\vphantom B}_k\cap  B_{\cos\phi_k}(0)\right)\cup\left(\vphantom{R_\beta^\beta}\widetilde E^{\vphantom B}_{k}\cap R_{|\beta-\alpha|,k}\right).
\]
Since $\cos\phi_k=0$, the first intersection is $\{0\}$. And $\widetilde E^{\vphantom B}_{k}$ consists of those $\mu$ with $\arg\mu\in(-\pi/2,\pi/2)$, that is with non-negative real part. As $\delta_k=\pi/2$, we have $\cot\delta_k=0$, and with arguments like those in the proof of  \cref{lemma: the cone} we get that $R_{|\beta-\alpha|,k}=(-\infty,|\beta-\alpha|]$. So $\widetilde E^{\vphantom B}_{k}\cap R_{|\beta-\alpha|,k}=[0,|\beta-\alpha|]$ and thus $\Lambda_k(T^0_{\alpha,\beta})=[0,|\beta-\alpha|]$.

When  $k=\frac{n+1}2> m$ and $|\beta-\alpha|\leq\cos\psi_{k,m}$, even though $k>m$ we have  $C_{k,m}=\emptyset$; then the exact reasoning from previous paragraph applies. 

\item[(\arabic{lista})]\stepcounter{lista} $k=\frac{n+1}2> m, \ |\beta-\alpha|>\cos\psi_{k,m}$: now $C_{k,m}\ne\emptyset$. From \cref{theorem: Lambda_k of Jordan plus I m} we have 
\[
\Lambda_k(T^0_{\alpha,\beta})=\left(\vphantom{R_\beta^\beta}\widetilde D^{\vphantom B}_k\cap  B_{\cos\phi_k}(0)\right)\cup\left(\vphantom{R_\beta^\beta}\widetilde E^{\vphantom B}_{k}\cap R_{|\beta-\alpha|,k}\cap B_{\cos\psi_{k,m}(0)}\right).
\]
As in the previous step, we get $\widetilde E^{\vphantom B}_{k}\cap R_{|\beta-\alpha|,k}=[0,|\beta-\alpha|]$, but now we also have to cut with $B_{\cos\psi_{k,m}(0)}$. So $\Lambda_k(T^0_{\alpha,\beta})=[0,\cos\psi_{k,m}]$. 

\item[(\arabic{lista})]\stepcounter{lista} $\tfrac {n+1}2<k\leq n$, $k\leq m$: We again apply \cref{theorem: Lambda_k of Jordan plus I m} to get 
\[
\Lambda_k(T^0_{\alpha,\beta})=\left(\vphantom{R_\beta^\beta}\widetilde D^{\vphantom B}_k\cap B_{\cos\phi_k}(0)\right)
\cup
\left(\vphantom{R_\beta^\beta}\widetilde E^{\vphantom B}_{k}\cap R_{|\beta-\alpha|,k}\right).
\]
 From $k>(n+1)/2$ we get that $\phi_k>\pi/2$, so $\cos\phi_k<0$. This makes $\widetilde D^{\vphantom B}_k\cap B_{\cos\phi_k(0)}=\emptyset$ and $\pm\pi/2\not\in D_k$. By \cref{lemma: the cone}, if $x+iy\in R_{|\beta-\alpha|,k}$, we have 
\begin{equation}\label{equation: inequality for step 6}
x\cos\t-y\sin\t\leq|\beta-\alpha|\cos\t,\ \ \t\not\in D_k. 
\end{equation}
With $\t=\pm\pi/2$, we get $\pm y\leq0$, so $y=0$. Now the inequality \eqref{equation: inequality for step 6} is $x\cos\t\leq|\beta-\alpha|\cos\t$ for all $\t\not\in D_k$. Since $\delta_k>\pi/2$, the set $[0,2\pi]\setminus D_k$ contains $\t$ with $\t<\pi/2$ and also $\t$ with $\t>\pi/2$. Using these $\t$ we get $x\leq|\beta-\alpha|$ and $-x\leq-|\beta-\alpha|$, so $x=|\beta-\alpha|$. Thus $R_{|\beta-\alpha|,k}=\{|\beta-\alpha|\}$ and so $\Lambda_k(T^0_{\alpha,\beta})=\{|\beta-\alpha|\}$.

\bigskip

When  $n<k\leq m$: \cref{lemma: lambda_k. k greater than m} gives us directly that $\Lambda_k(T^0_{\alpha,\beta})=\{|\beta-\alpha|\}$. 

\bigskip

When $\tfrac {n+1}2<k$, $k>m$, $|\beta-\alpha|\leq\cos\psi_{k,m}$: Assume first that $k\leq n$. From \cref{theorem: Lambda_k of Jordan plus I m}, and noting that $\cos\phi_k<0$, we have 
\[\
\Lambda_k(T^0_{\alpha,\beta})=\widetilde E^{\vphantom B}_{k}\cap R_{|\beta-\alpha|,k}\cap B_{\cos\psi_{k,m}}.
\] Using, as above, that $\pm\pi/2\not\in D_k$, we get that $R_{|\beta-\alpha|,k}=\{|\beta-\alpha|\}$. As $|\beta-\alpha|\leq\cos\psi_{k,m}$, we have $\Lambda_k(T^0_{\alpha,\beta})=\{|\beta-\alpha|\}$.

When $k>n$, $k>m$, and $|\beta-\alpha|\leq\cos\psi_{k,m}$, \cref{lemma: lambda_k. k greater than m} gives us the result.

\item[(\arabic{lista})]\stepcounter{lista}  $\tfrac {n+1}2<k $, $k>m$, $|\beta-\alpha|>\cos\psi_{k,m}$: Assume first that $k\leq n$. As in the previous cases, the only possible value for $x$ is $|\beta-\alpha|$. But now the condition $|\beta-\alpha|>\cos\psi_{k,m}$ means that $|\beta-\alpha|\not\in B_{\cos\psi_{k,m}}(0)$, so by \cref{theorem: Lambda_k of Jordan plus I m} we have $\Lambda_k(T^0_{\alpha,\beta})=\emptyset$. 

When $k>n$, \cref{lemma: lambda_k. k greater than m} gives us the result. \qedhere
\end{enumerate}
\end{proof}

Now we can do the rotated and translated version of \cref{proposition: Omegak of T alpha beta m}. For this 
we consider the translated and rotated versions of $\widetilde E^{\vphantom B}_{k}$ and $R_{r,k}$, 
\[
\widetilde E^{\psi}_k =\alpha+e^{i\psi}\widetilde E^{\vphantom B}_{k},\ \ \ \ R_{r,k}^\psi=\alpha+e^{i\psi}R_{r,k}. 
\]
We will use the notation 
\[
[\alpha,\beta]=\{\alpha+t(\beta-\alpha):\ t\in[0,1]\}=\{(1-t)\alpha+t\beta:\ t\in[0,1]\}. 
\]
Finally, we get to write explicitly the higher rank numerical ranges of $J_n(\alpha)\oplus\beta I_m$. 

\begin{theorem}\label{theorem: Omegak of T alpha beta m}
Let $T=J_n(\alpha)\oplus \beta I_m$. Let $k\in\{1,\ldots,n+m\}$. Put $\psi=\arg(\beta-\alpha)$. Then 
$\Lambda_k(T)$ is expressed by the following table: \\[0.1cm]
\tcbset{tab2/.style={enhanced,fonttitle=\bfseries,fontupper=\normalsize\sffamily,
colback=yellow!10!white,colframe=red!50!black,colbacktitle=Salmon!30!white,
coltitle=black,center title}}

\renewcommand{\arraystretch}{1.5}
\setcounter{lista}{1}

\newsavebox{\tablateorema}
\savebox{\tablateorema}{
\begin{tabular}{rll}
\multicolumn{1}{c}{ }&\multicolumn{1}{c}{\text{$\Lambda_k(T)$}} &\multicolumn{1}{c}{ \textrm{Condition}  }       \\\hline  \\[-0.1cm]
\arabic{lista}\stepcounter{lista}& $B_{\cos\phi_k}(\alpha)$ & $\ 1\leq k\leq\frac n2,\   |\beta-\alpha|\leq\cos\phi_k$ \ \\[0.3cm] \hline
\arabic{lista}\stepcounter{lista}& $B_{\cos\phi_k}(\alpha)\cup (\widetilde E^{\psi}_k \cap R^\psi_{|\beta-\alpha|,k})$&$\ 1\leq k\leq\frac n2, \ k\leq m,\ |\beta-\alpha|>\cos\phi_k$ \ \\[0.3cm]  \hline
\arabic{lista}\stepcounter{lista}& $B_{\cos\phi_k}(\alpha)\ \cup \left(\vphantom{R_\beta^\beta}\widetilde E^{\psi}_k \cap R^\psi_{|\beta-\alpha|,k}\cap B_{\cos\psi_{k,m}}(\alpha)\right)$ & $\ 1\leq k\leq\frac n2,\  k> m,\ |\beta-\alpha|>\cos\phi_k$ \ \\[0.3cm]  \hline
\arabic{lista}\stepcounter{lista}& $[\alpha,\beta]$ & $\ k=\frac{n+1}2\leq m$,\ \text{ or}  \ \\[0.3cm] 
&& $\ k=\frac{n+1}2> m,\ |\beta-\alpha|\leq\cos\psi_{k,m}$ \ \\[0.3cm]  \hline
\arabic{lista}\stepcounter{lista}& $\{\alpha+t(\beta-\alpha)\cos\eta_{k,m}:\ t\in[0,1]\}$ & $\ k=\frac{n+1}2> m, \ |\beta-\alpha|>\cos\psi_{k,m}$ \ \\[0.3cm]  \hline
\arabic{lista}\stepcounter{lista}& $\{\beta\}$ & $\ \frac{n+1}2<k\leq  m$, \ \text{ or}  \ \\[0.3cm]  
& & $\ \frac{n+1}2<k,\  m<k,\ |\beta-\alpha|\leq\cos\psi_{k,m}$ \ \\[0.3cm] \hline
\arabic{lista}\stepcounter{lista}& $\emptyset$ & $\ \frac{n+1}2<k,\ m<k,\ |\beta-\alpha|>\cos\psi_{k,m}$ \ \\[0.3cm]
\end{tabular}
} 
\begin{center}
\resizebox{14cm}{!}{
\begin{tcolorbox}[tab2,tabularx={rll},title={$T=J_n(\alpha)\oplus \beta I_m$},boxrule=0.5pt,width=\wd\tablateorema]
\usebox{\tablateorema}
\end{tcolorbox}
} 
\end{center}
\end{theorem}
\begin{proof}
We have $T=\alpha I_{n+m}+e^{i\psi}\,T^0_{\alpha,\beta}$. So $\Lambda_k(T)=\alpha+e^{i\psi}\Lambda_k(T^0_{\alpha,\beta})$. Thus the result is a direct application of \cref{proposition: Omegak of T alpha beta m}. Note that $\alpha+e^{i\psi}|\beta-\alpha|=\alpha+\beta-\alpha=\beta$, and 
\[
\alpha+e^{i\psi}[0,|\beta-\alpha|]=\{\alpha+t\,e^{i\psi}|\beta-\alpha|:\ t\in[0,1]\}
=\{\alpha+t(\beta-\alpha):\ t\in[0,1]\}. 
\]
Also,
\[
\alpha+te^{i\psi}\cos\psi_{k,m}=\alpha+te^{i\psi}|\beta-\alpha|\cos\eta_{k,m}
=\alpha+t(\beta-\alpha)\cos\eta_{k,m}.\qedhere
\]
\end{proof}

\begin{examples}\label{remark: examples}
We include a few graphic examples of $\Lambda_k(J_n(\alpha)\oplus \beta I_m)$. The graphs were produced with a Javascript program that draws the lines $x\cos\t-y\cos\t=\lambda_k(T)$ for $\t$ ranging (in degrees) from 1 to 359. This is not always an accurate representation, because in some cases the intersection of the semiplanes is empty but the lines still leave a clearly unshaded region; for this we produced a version of the script that indicates  the semiplanes instead of just drawing the lines. This issue does not make an appearance in the examples we included. The tool is available upon request. 

We can see in these pictures the situation described in \cref{proposition: understanding the shape 2,proposition: understanding the shape}.

\begin{enumerate}
\item In \cref{n4i4k1}, the unshaded region represents $\Lambda_1(J_4(0)\oplus I_4)$. In \cref{n4i4k2} we see $\Lambda_2(J_4(0)\oplus I_4)$. Grid lines are set on integer multiples of $0.2$. 
\item In \cref{n5i5k1}, we have $\Lambda_1(J_5(-1-i)\oplus (1-2i)I_5)$, and in \cref{n5i5k2}, we have $\Lambda_2(J_5(-1-i)\oplus (1-2i)I_5)$
\end{enumerate}
\end{examples}

\begin{figure} 
\centering
\begin{minipage}{0.47\textwidth}
\includegraphics[width=\mytextwidth]{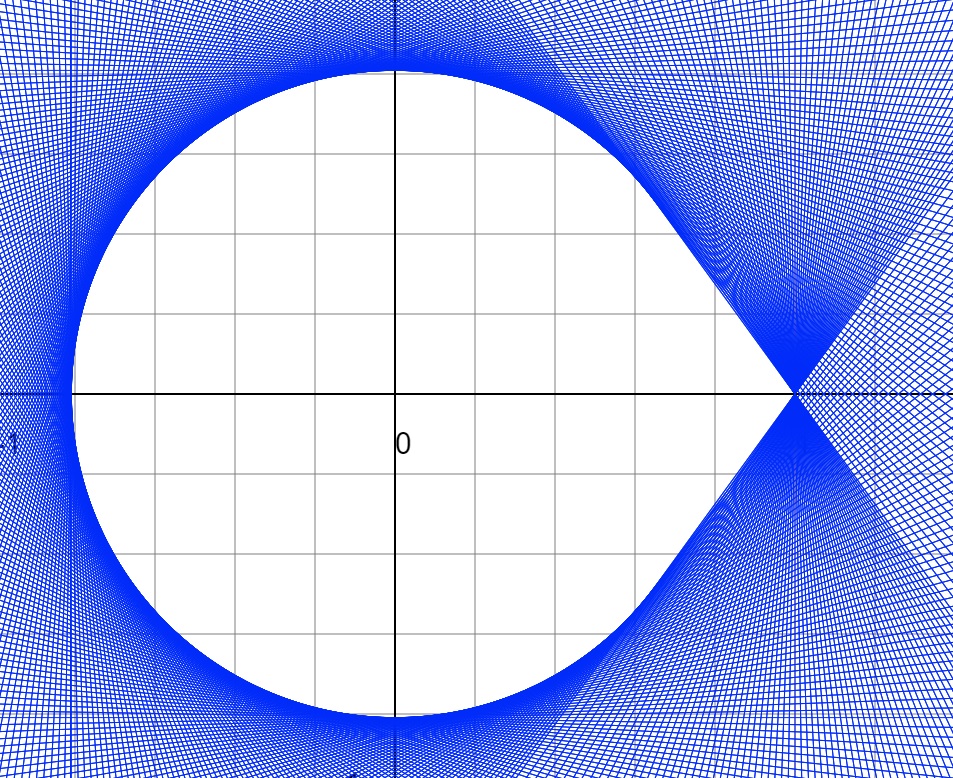}

\captionsetup{width=\textwidth,justification=centering}
\caption{$n=4$, $m=4$, $k=1$,  $\alpha=0$, $\beta=1$}
\label{n4i4k1}

\end{minipage}\hfill
\begin{minipage}{0.47\textwidth}

\includegraphics[width=\mytextwidth]{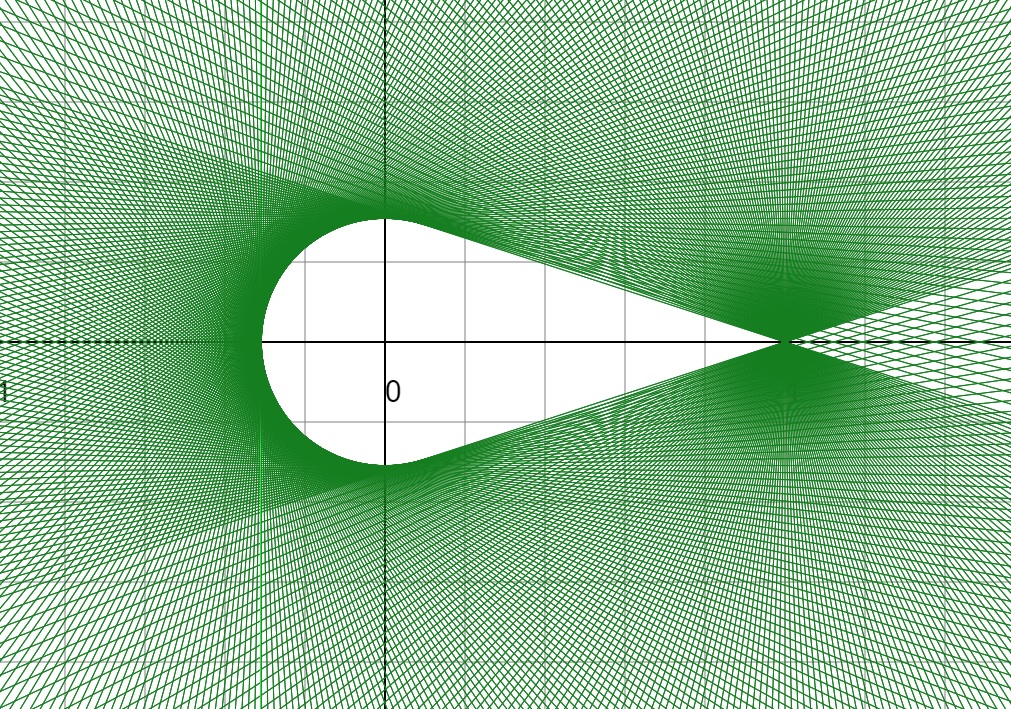}

\captionsetup{width=\textwidth,justification=centering}
\caption{$n=4$, $m=4$, $k=2$, $\alpha=0$, $\beta=1$}
\label{n4i4k2}
\end{minipage}
\end{figure}

\begin{figure}
\centering
\begin{minipage}{0.47\textwidth}
\includegraphics[width=\mytextwidth]{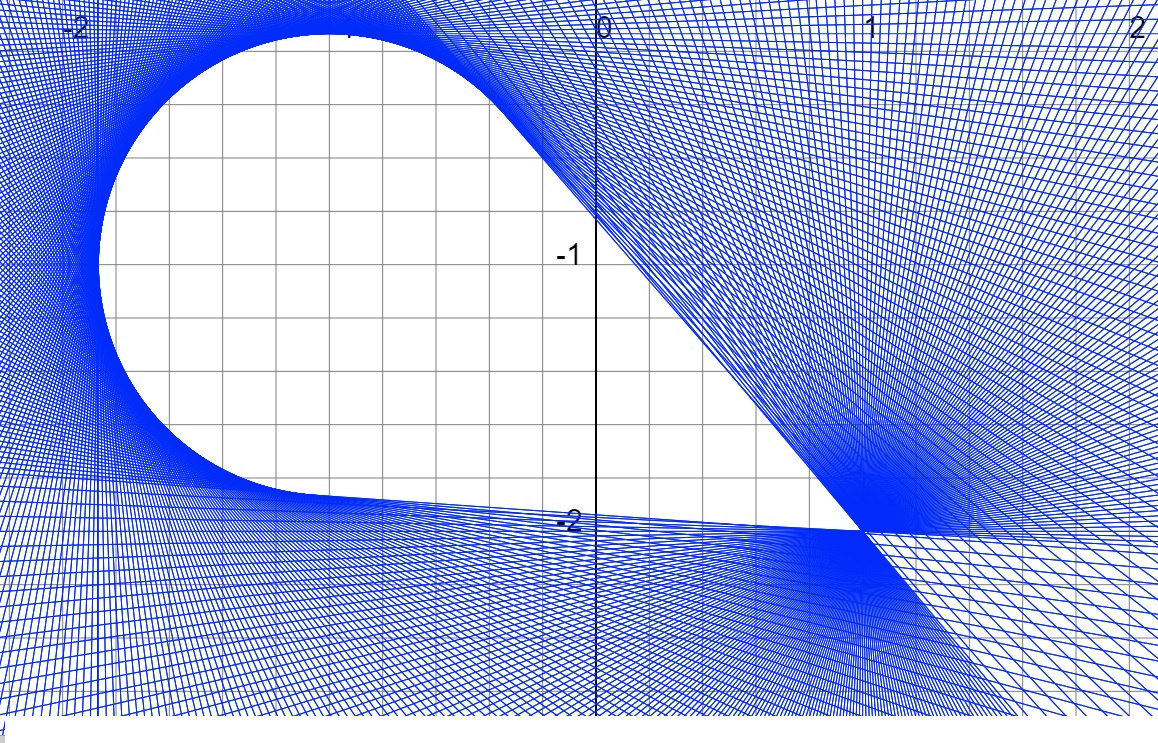}

\captionsetup{width=\textwidth,justification=centering}
\caption{$n=5$, $m=5$, $k=1$, $\alpha=-1-i$, $\beta=1-2i$}
\label{n5i5k1}
\end{minipage}\hfill
\begin{minipage}{0.47\textwidth}
\includegraphics[width=\mytextwidth]{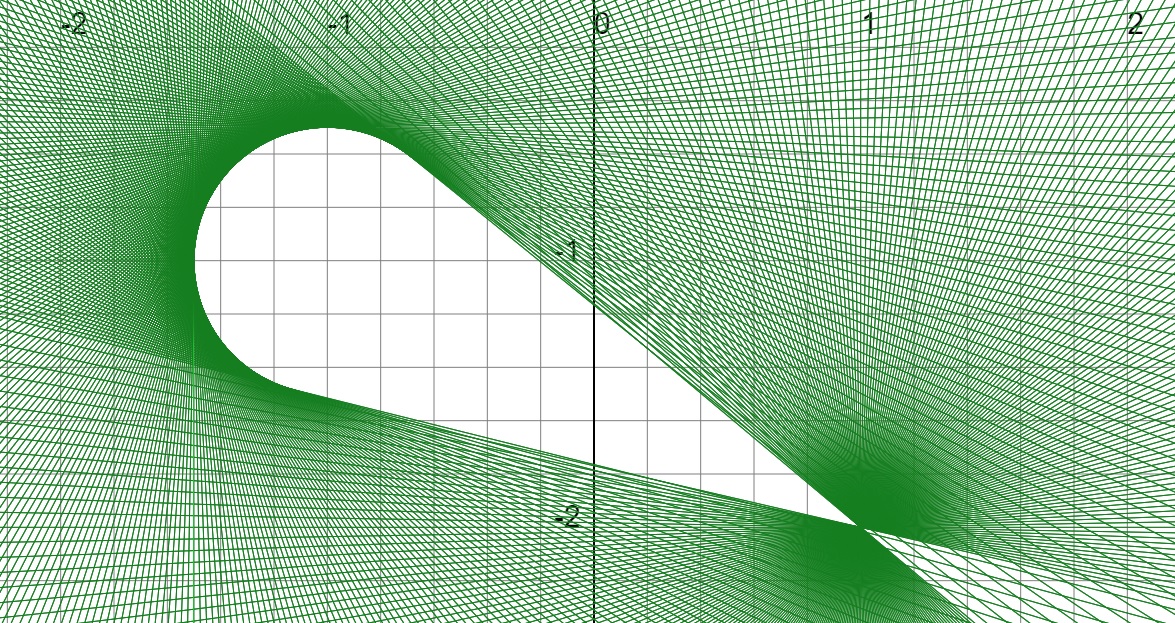}

\captionsetup{width=\textwidth,justification=centering}
\caption{$n=5$, $m=5$, $k=2$, $\alpha=-1-i$, $\beta=1-2i$}
\label{n5i5k2}
\end{minipage}
\end{figure}

\begin{remark}\label{remark: new radius}
When $m<n$, a new radius, $\cos\psi_{k,m}$, makes an appearance if $m<k\leq n/2$. In \cref{n5i5k1} this does not occur, but it does in \cref{n5i1k2}, for $\Lambda_2(J_5(-1-i)\oplus (1-2i))$. This is a case where $\Lambda_k(T)$ is not a convex combination of certain (higher) numerical ranges of its direct summands. In \cref{n5i1k2BC} we can see a representation of the (areas corresponding to the) sets $D_k$---in blue---and $C_{k,m}$---in red. 
\end{remark}

\begin{figure}
\centering
\begin{minipage}{0.47\textwidth}
\includegraphics[width=\mytextwidth]{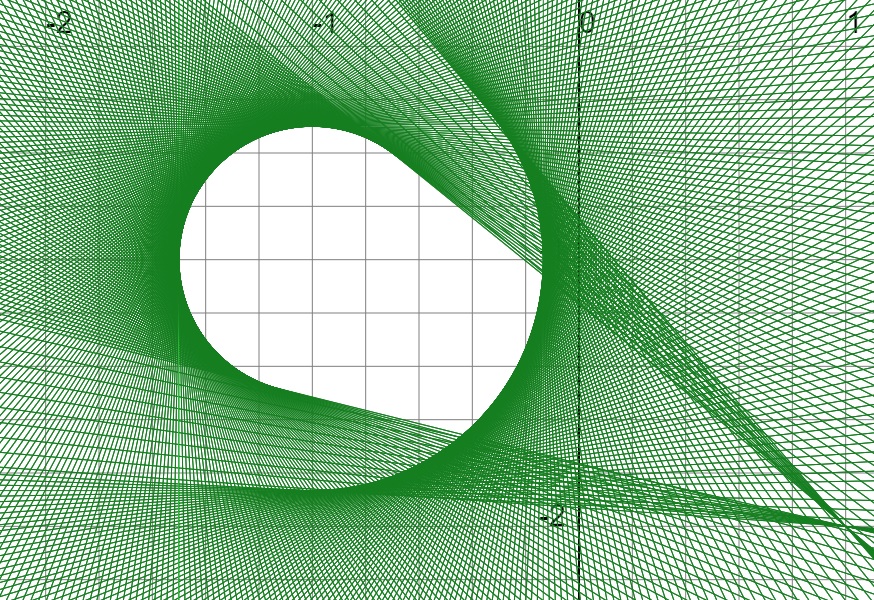}
\captionsetup{width=\textwidth,justification=centering}
\caption{$n=5$, $m=1$, $k=2$, $\alpha=-1-i$, $\beta=1-2i$}
\label{n5i1k2}
\end{minipage}\hfill
\begin{minipage}{0.47\textwidth}
{\begin{tikzpicture}[thick,scale=1, every node/.style={transform shape}]
    \node[anchor=south west,inner sep=0] at (0,0) 
{\includegraphics[width=\mytextwidth]{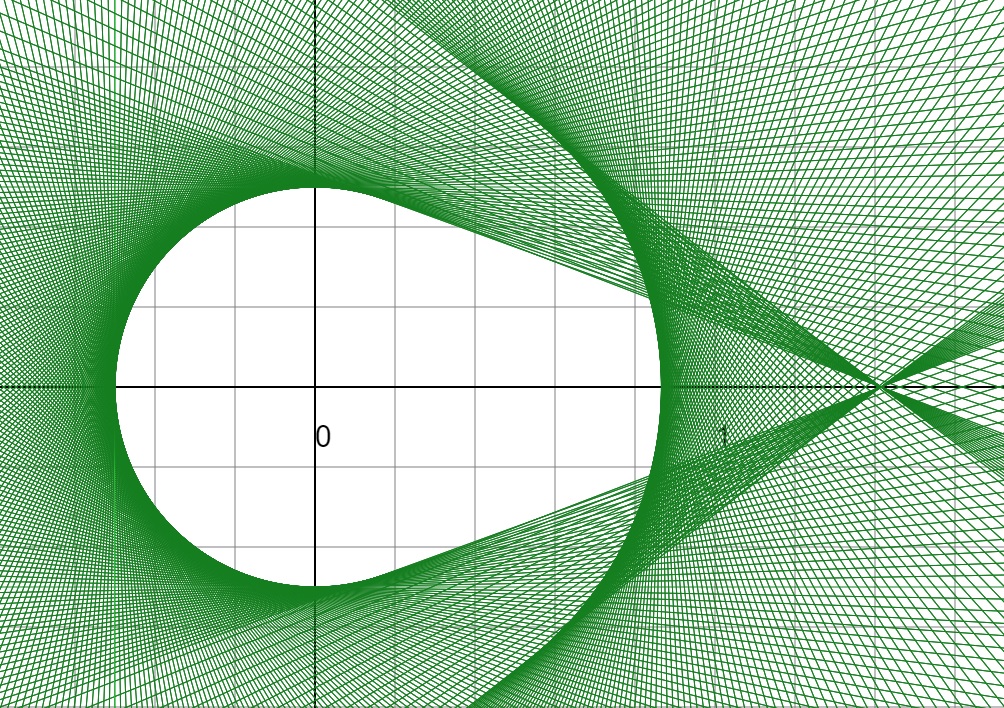}};

\node[anchor=south west,inner sep=0] at (0.82,0.87) {
    \begin{tikzpicture}[scale=0.47, every node/.style={scale=0.6}]
    \coordinate (centre) at (4.75,4.83);

     \draw[color=red,thick] (centre) -- (9.77,6.17);
     \draw[color=red,thick] (centre) -- (9.77,3.49);
     \draw[color=blue,thick] (centre) -- (5.75,7.65);
     \draw[color=blue,thick] (centre) -- (5.75,2.01);
     \draw[color=yellow] (centre) circle [radius=2.98];
     \draw[color=blue, thick] ($ (centre) !0.5! (5.75,7.65)$) arc [radius=1.49, start angle=70, end angle=290];
     \node[color=blue] at  ($(centre) + (-2,0.5) $)   {$D_k$};
     \node[color=red] at  ($(centre) + (3,-0.5) $)   {$C_{k,m}$};
     \draw[color=red,thick] (8.4,3.84) arc [radius=4.2, start angle=-15, end angle=12];
    \end{tikzpicture}
};     
\end{tikzpicture}
}
\captionsetup{width=\textwidth,justification=centering}
\caption{$n=5$, $m=1$, $k=2$, $\alpha=0$, $\beta=\sqrt2$}
\label{n5i1k2BC}
\end{minipage}
\end{figure}

\FloatBarrier

\section{Remarks and Applications}\label{section: remarks}

\begin{remark}\label{remark: corner point}
The results in \cref{theorem: Omegak of T alpha beta m} and the accompanying images show concrete examples of the following result of Chang, Gau, and Wang (here $W_k(T)$ denotes Halmos' higher numerical range): 

\begin{proposition}[\cite{ChangGauWang-EqualityHigherRankNumericalRanges2014}]
Let $T\in M_n(\CC)$, $k\in\{1,\ldots,n\}$, and $\beta\in \Lambda_1(T)$ a point that is a corner. The following statements are equivalent: 
\begin{enumerate}
\item $\beta$ is a corner of $W_k(T)$;
\item $\beta$ is a corner of $\Lambda_k(T)$;
\item $T$ is unitarily equivalent to $\beta I_m\oplus C$, with $m\geq k$ and $\beta\not\in \Lambda_1(C)$. 
\end{enumerate}
\end{proposition}

In particular, Figure \ref{n5i1k2} shows an example of how the corner $\beta$ can disappear as soon as $k>m$. We note also that in the same example two new corners appear in $\Lambda_2(T)$, which are not eigenvalues. Thus there is no Donoghue's Theorem \cite{donoghue1957numerical} for $k\geq2$. 
\end{remark}

\begin{remark}\label{remark: extreme nonempty}
It was proven in \cite[Proposition 2.2]{choi2006higherB} that $\Lambda_k(T)$ is at most a singleton when $k>n/2$. In the opposite direction, it was shown in \cite{LiPoonTse-ConditionNonEmpty_2009} that, for $T\in M_n(\CC)$, $\Lambda_k(T)$ is always nonempty if $k<n/3+1$, and that it can be empty as early as $k=n/3+1$ in specific examples.
The example they give is of a normal operator, and they mention that their example can be perturbed to obtain a non-normal example. Here, \cref{theorem: Omegak of T alpha beta m} gives us a natural non-normal example. Indeed, in the context of \cref{theorem: Omegak of T alpha beta m} their $n$ becomes $n+m$; if $k=(n+m)/3+1$ and $m>(n-3)/2$, then 
\[
k=\frac{n+m}3+1>\frac{n+\frac{n-3}2}3+1=\frac{n-1}2+1=\frac{n+1}2.
\]
If we also require $m<\tfrac{n+1}2$, it follows that $k>m$.  
Taking $\alpha=0$, $\beta\geq1$, condition (7) in \cref{theorem: Omegak of T alpha beta m} guarantees that $\Lambda_k(J_n(0)\oplus \beta I_m)=\emptyset$. So, for instance, with $n=4$, $m=2$, $k=(n+m)/3+1=3$  we have that $\Lambda_3(J_4(0)\oplus I_2)=\emptyset$ and $3=k=6/3+1$. Or, for another example, $\Lambda_5(J_8(0)\oplus I_4)=\emptyset$, where $k=5=\frac{12}3+1$.

It is also possible to find cases where our examples have nonempty $\Lambda_k(T)$ for fairly big $k$. Most examples in the literature of these extreme situations are normal, while---as we mentioned---ours are non-normal. One straightforward way to force the issue is to take very large $m$ (the size of the scalar block) as then  we will always have $\Lambda_k(T^0_{\alpha,\beta})\ne\emptyset$ for $k=m$. But nonempty higher rank numerical ranges for big $k$ appear in our examples even without the need of a  big $m$ relative to $n$. 

We see from \cref{theorem: Omegak of T alpha beta m} that $\Lambda_k(J_n(\alpha)\oplus \beta I_m)=\emptyset$ when $k>m$ and $\cos\psi_{k,m}<0$. The condition $\cos\psi_{k,m}\geq0$ is $\frac{(k-m)\pi}{n+1}\leq\frac\pi2$, which we write as  $k\leq m+\tfrac{n+1}2$. When $n$ is odd and $k=m+\tfrac{n+1}2$, we have $\cos\psi_{k,m}=0$. So to have $\Lambda_k(J_n(\alpha)\oplus \beta I_m)\ne\emptyset$ with the biggest possible $k$, we need (by 6 and 7 in \cref{theorem: Omegak of T alpha beta m}) that $|\beta-\alpha|=0$. We also need $k\leq n+m-1$ as $\Lambda_{n+m}(T^0_{\alpha,\beta})=\emptyset$. The condition $m+(n+1)/2\leq m+n-1$ forces $n\geq3$ and the equality can only occur when $n=3$. 

To see an example of this consider $T=J_3(0)\oplus 0_m\in M_{3+m}(\CC)$. If we take $k=2+m$, then $k>m$ and $\cos\psi_{k,m}=0$. As $|\beta-\alpha|=0$, we get from \cref{theorem: Omegak of T alpha beta m} that $\Lambda_{2+m}(T)=\{0\}$. An explicit rank-$(m+2)$ projection $P$ with $P(J_3(0)\oplus 0_m)P=0$ is given by 
\[
P=\begin{bmatrix} 1&0&0\\0&0&0\\0&0&1\end{bmatrix}\oplus I_m. 
\]

Similarly, consider $n=5$. Now $\tfrac {n+1}2=3<4=n-1$. Since $\cos\psi_{4+m}=\cos \tfrac {4\pi}6=-\tfrac 12$, we have that $\Lambda_{4+m}(J_5(\alpha)\oplus \beta I_m)=\emptyset$ for any $\alpha,\beta$. But $\Lambda_{3+m}(J_5(0)\oplus 0_m)=\{0\}$ by case 4 in \cref{theorem: Omegak of T alpha beta m}. As $\Lambda_3(J_5(0))=\{0\}$, it is enough to find a projection $Q\in M_5(\CC)$, of rank 3, such that $QJ_5(0)Q=0$. An easy concrete realization of such $Q$ is 
\[
Q=\begin{bmatrix} 
1&0&0&0&0\\
0&0&0&0&0\\
0&0&1&0&0\\
0&0&0&0&0\\
0&0&0&0&1\\
\end{bmatrix}.
\]
If we put  $P=Q\oplus I_m$, then $P$ is a projection of rank $3+m$ and we have $P(J_5(0)\oplus 0_m)P=0_{5+m}$. 

In general, if $n=2\ell+1$, then $\Lambda_{\ell+1}(J_{2\ell+1}(0))=\{0\}$ and we can form $Q=\sum_{j=1}^{\ell+1}E_{2j-1,2j-1}$ to get a rank-$(\ell+1)$ projection $Q$ with $QJ_{2\ell+1}(0)Q=0$. Indeed, 
\[
QJ_{2\ell+1}(0)Q=\sum_{j,h=1}^{\ell+1}\sum_{k=1}^{2\ell+1}E_{2j-1,2j-1}E_{k,k+1}E_{2h-1,2h-1}=0,
\]
since $k$ and $k+1$ cannot be both odd. Then $P=Q\oplus I_m$ is a rank-$(\ell+1+m)$ projection with $P(J_{2\ell+1}(0)\oplus 0_m)P=0_{2\ell+1+m}$, showing explicitly (note that it also follows directly from case 6 in \cref{theorem: Omegak of T alpha beta m}) that $\Lambda_{\ell+1+m}(J_{2\ell+1}(0)\oplus 0_m)=\{0\}$. For $k=\ell+2+m$, we have $\cos\psi_{\ell+2+m,m}=\cos\frac{(\ell+2)\pi}{2\ell+2}<0$, so $\Lambda_{\ell+2+m}(J_{2\ell+1}(0)\oplus 0_{m})=\emptyset$ by case 7 in \cref{theorem: Omegak of T alpha beta m}.
\end{remark}

\begin{question}
The only way to have $T\in M_n(\CC)$ with $\Lambda_n(T)\ne\emptyset$ is to have $T=\beta I$ for some $\beta$. We see from \cref{remark: extreme nonempty} that $\Lambda_{2+m-1}(J_2(0)\oplus 0_m)\ne\emptyset$, and $\Lambda_{3+m-1}(J_3(0)\oplus 0_m)\ne\emptyset$, while $\Lambda_{n+m-1}(J_n(\alpha)\oplus \beta I_m)=\emptyset$ for $n\geq4$ and any $\alpha,\beta$. 
This suggests the following question:
Given $n\geq4$, does there exist non-normal $T\in M_n(\CC)$ with $\Lambda_{n-1}(T)\ne\emptyset$? The existence of a normal irreducible $T\in M_n(\CC)$, not a scalar multiple of the identity, with $\Lambda_{n-1}(T)\ne\emptyset$  was established in \cite[Theorem 3]{LiPoonTse-ConditionNonEmpty_2009}.
\end{question}

\begin{remark}
The following result due to J. Anderson. There is a nice proof in  \cite{TamYang-CircularSymmetry1999}, where it is attributed to Pei-Yuan Wu (who published more general results in \cite{Wu2011}). An infinite-dimensional version appears in \cite{gau2006anderson}, where the authors briefly discuss the story of the theorem and the various proofs that have been published.

\begin{proposition}[J. Anderson]\label{proposition: numerical range is a disk}
Let $n\geq 2$ and $T\in M_n(\CC)$. Let $\alpha\in\CC$, $r>0$. If $\Lambda_1(T)\subset B_r(\alpha)$ and $|\Lambda_1(T)\cap \partial B_r(\alpha)|\geq n+1$, then $\Lambda_1(T)=B_r(\alpha)$. 
\end{proposition}

One might be tempted to try to think of $\Lambda_k(T)$ as the numerical range of some amplification/dilation of $T$. Actually, this works for some of our examples: for instance, we see from \cref{theorem: Omegak of T alpha beta m} that $\Lambda_2(J_5(0)\oplus I_4)=\conv(B_{1/2}(0)\cup\{1\})=\Lambda_1(J_2(0)\oplus I_4)$. But, at the same time  \cref{proposition: numerical range is a disk}, together with some of our examples above show that this is not the case in general. Concretely, if we look at the example from \cref{n5i1k2BC}, namely $\Lambda_2(J_n(-1-i)\oplus (1-2i)I_1)$, we can see that the whole set is contained in the disk of radius $\cos\psi_2=\cos\tfrac{(2-1)\pi}{5+1}=\tfrac{\sqrt3}2$ centered at $-1-i$, and it shares a nontrivial part of the arc; thus \cref{proposition: numerical range is a disk} implies that $\Lambda_2(J_n(-1-i)\oplus (1-2i)I_1)$ is not the numerical range of any matrix. We also conclude that there is no analogue of \cref{proposition: numerical range is a disk} for any $k>1$.

\end{remark}

\begin{remark}
An easy and well-known property of the higher numerical range is that 
\begin{equation}\label{equation: higher of direct sum}
\Lambda_k(T\oplus S)\supset \Lambda_k(T)\cup\Lambda_k(S).
\end{equation}
 As $\Lambda_k(T\oplus S)$ is convex, it will always contain $\conv\{\Lambda_k(T)\cup\Lambda_k(S)\}$. But it is often the case that the inclusion is strict, as for instance when $\Lambda_k(T\oplus T)$ with $T\in M_n(\CC)$ and $k>n$. In \cref{theorem: Omegak of T alpha beta m}, case 2,  we see that the inclusion \eqref{equation: higher of direct sum} can be an equality for several values of $k$; indeed, under the conditions of case 2 we have $\Lambda_k(J_n(\alpha))=B_{\cos\phi_k}(\alpha)$, and $\Lambda_k(\beta I_m)=\{\beta\}$ and $\Lambda_k(J_n(\alpha)\oplus\beta I_m)=\conv\Lambda_k(J_n(\alpha))\cup \Lambda_k(\beta I_m)$. 
\end{remark}

\begin{remark}
If $T_1,T_2$ are unitarily equivalent, then $\Lambda_k(T_1)=\Lambda_k(T_2)$ for all $k$. The converse is known to be false in general \cite{GauWu2013KippenhahnPolynomials}. The class of matrices of the form $J_n(\alpha)\oplus \beta I_m$ is rigid enough that the family of higher rank numerical ranges characterizes unitary equivalence (equality, actually). Namely, 
\begin{corollary}\label{corollary: unitary equivalence}
Let $T_j=J_{n_j}(\alpha_j)\oplus\beta_j I_{m_j}$, $j=1,2$, such that $n_1+m_1=n_2+m_2$ and such that for all $k$ we have $\Lambda_k(T_1)=\Lambda_k(T_2)$. Then $T_1=T_2$. 
\end{corollary}
\begin{proof} 
We refer to the cases that appear in the table in \cref{theorem: Omegak of T alpha beta m}. Consider first $k=1$. From \cref{theorem: Omegak of T alpha beta m} we know that both $T_1,T_2$ fall in the same of cases 1 or 2. In both cases we have that part of the boundary of $\Lambda_1(T_j)$ is an arc of a circle of radius $\cos\phi_k$ centered at $\alpha_j$ (the number $\cos\phi_k$ is in principle different for $T_1$ and $T_2$, but since we are arguing that in this case it is the same for both, there is no need for a particular notation for that). Thus $\alpha_1=\alpha_2$, and looking at the cosines we need $1/(n_1+1)=1/(n_2+1)$, so $n_1=n_2$ and then $m_1=m_2$. 

If any of cases 2 or 3 arise for some $k$, 
as the (extensions of the, in case 3)  line segments intercept at $\beta$ (recall that $R^\psi_{|\beta-\alpha|,k}=\alpha+e^{i\psi}R_{|\beta-\alpha|,k}$ and $\alpha+e^{i\psi}|\beta-\alpha|=\beta$), we get that $\beta_1=\beta_2$. 

If neither case 2 nor 3 arises, we are in case 1 for all $1\leq k\leq n/2$ for both $T_1$ and $T_2$. So $|\beta_1-\alpha|$,$|\beta_2-\alpha|\leq\cos\phi_k<1$ for all such $k$. Thus
\begin{equation}\label{equation: condition when neither 2 nor 3}
|\beta_j-\alpha|\leq\cos\frac{\lfloor n/2\rfloor}{n+1}\,\pi.
\end{equation}
If case 6 arises for some $k$, we get $\beta_1=\beta_2$. And case 6 will always arise in the presence of \eqref{equation: condition when neither 2 nor 3}; for if case 7 occurs already for $k=\lfloor n/2\rfloor + 1$, we have $m_j<\lfloor n/2\rfloor + 1$ so
\[
k-m_j=\lfloor \tfrac n2\rfloor + 1-m_j\leq\lfloor\tfrac n2\rfloor
\]
and thus
\[
|\beta_j-\alpha|>\cos\psi_{k,m_j}=\cos \frac{\lfloor n/2\rfloor + 1-m_j}{n+1}\,\pi
\geq \cos\frac{\lfloor n/2\rfloor}{n+1}\,\pi\geq|\beta_j-\alpha|,
\]
a contradiction. 
\end{proof}
\end{remark}

\section{Acknowledgements}

This work has been supported in part by the Discovery Grant program of the Natural
Sciences and Engineering Research Council of Canada grant RGPIN-2015-03762.

\end{document}